\documentclass[11pt]{amsart}
\usepackage{amssymb}

\newcommand{\QQ}{\mathbb{Q}}
\newcommand{\ZZ}{\mathbb{Z}}
\newcommand{\CC}{\mathbb{C}}
\newcommand{\RR}{\mathbb{R}}

\newcommand{\PP}{\mathbb{P}}

\renewcommand{\AA}{\mathbb{A}}

\newcommand{\Gal}{\operatorname{Gal}}
\newcommand{\PSL}{\operatorname{PSL}}

\newcommand{\lcrit}[1]{\lambda_{\mathrm{crit}, #1}}
\newcommand{\hcrit}{h_{\mathrm{crit}}}
\newcommand{\lhat}{\hat{\lambda}}
\newcommand{\hhat}{\hat{h}}
\newcommand{\hnaive}[1]{h_{\mathrm{ #1}}}

\renewcommand{\epsilon}{\varepsilon}

\newtheorem{theorem}{Theorem}
\newtheorem{corollary}[theorem]{Corollary}
\newtheorem{lemma}[theorem]{Lemma}

\title{A finiteness result for post-critically finite polynomials}
\author{Patrick Ingram}
\address{Department of Pure Mathematics, University of Waterloo}
\email{pingram@math.uwaterloo.ca}
\thanks{The author would like to thank Rafe Jones and Joseph Silverman for useful conversations which contributed to these results, and the anonymous referees for their many helpful comments and corrections. This research is supported in part by a grant from NSERC of Canada.}

\begin{document}
\maketitle

\begin{abstract}
Let $\mathcal{P}_d$ denote the moduli space of polynomials of degree $d$, up to affine conjugacy. We show that the set of points in $\mathcal{P}_d(\CC)$ corresponding to post-critically finite polynomials is a set of algebraic points of bounded height.  It follows that for any $B$, the set of conjugacy classes of post-critically finite polynomials of degree $d$ with coefficients of algebraic degree at most $B$ is a finite and effectively computable set.  As an example, we exhibit a complete list of representatives of the conjugacy classes of monic post-critically finite cubic polynomials in $\QQ[z]$.  The proof of the main result comes down to finding a relation between the natural height on $\mathcal{P}_d$, and Silverman's critical height.
\end{abstract}

\section{Introduction}

%\address{Department of Pure Mathematics, University of Waterloo}
%\email{pingram@math.uwaterloo.ca}
%\date{\today}

%\newcommand{\nts}[1]{{\bf\large *** #1 ***}}

%\newtheorem{theorem}{Theorem}
%\newtheorem*{thm}{Theorem}
%\newtheorem{proposition}[theorem]{Proposition}
%\newtheorem{lemma}[theorem]{Lemma}
%\newtheorem*{conj}{Conjecture}
%\newtheorem{corollary}[theorem]{Corollary}

%\theoremstyle{remark}
%\newtheorem*{remark}{Remark}

%\begin{document}
%\begin{abstract}

%\end{abstract}

%

%
%\maketitle

%

%\section{Introduction}

Many of the dynamical properties of a polynomial $f(z)\in\CC[z]$ may be deduced from properties of the forward orbits of the critical points.  For example, the Julia set of $f$ is connected if and only if these orbits are all bounded in the complex plane.  One special case in which this happens is when the orbit of the critical point is in fact finite, and in general we will call any rational map  $f:\PP^1_\CC\to\PP^1_\CC$ \emph{post-critically finite} if and only if is has this property.

One broad class of examples of postcitically finite maps are the Latt\`{e}s maps, obtained by descending an endomorphism of a torus to the projective line.  Other than these examples, post-critically finite maps are relatively sparse.
Let $\mathcal{M}_d$ denote the moduli space of rational functions on $\PP^1_\CC$, up to $\PSL_2$-conjugacy.  A result of Thurston \cite{thurston} implies that the non-Latt\`{e}s post-critically finite maps are covered by a countable union of 0-dimensional subvarieties of $\mathcal{M}_d(\CC)$.  Since these varieties are all defined over $\QQ$, one sees immediately that Thurston's theorem ensures that every non-Latt\`{e}s post-critically finite map is defined over $\overline{\QQ}$, up to a change of variables.  But since $\mathcal{M}_d(\overline{\QQ})$ is a countable set, it is still \emph{prima facie} possible that every map defined over $\overline{\QQ}$ is post-critically finite. 

In the case of quadratic polynomials, we may show directly that any value $c\in\CC$ for which the corresponding polynomial $z^2+c$ is post-critically finite must be an algebraic integer (since there exist distinct $m, n\in\mathbb{N}$ such that $f^n(0)=f^m(0)$, and this translates into a non-trivial monic polynomial condition on $c$ over $\ZZ$).  What's more, these values of $c$, and all of their Galois conjugates, must be contained in the Mandelbrot set (the collection of $c\in\CC$ for which $z^2+c$ is post-critically \emph{bounded}), which in turn is contained in the disk of radius 2. Hence, each such value of $c$ has absolute logarithmic height at most $\log(2)$, and Northcott's Theorem then tells us that for any $B\geq 1$, the set of values $c\in\overline{\QQ}$ such that $z^2+c$ is post-critically finite and $[\QQ(c):\QQ]\leq B$, is a finite and effectively computable set.  More than simply being sparse in the complex setting, post-critically finite quadratic polynomials are fairly rare even within the realm of polynomials with algebraic coefficients.

The purpose of this note is to extend this observation to polynomials of arbitrary degree. 
At the Bellaires workshop of 2010, Silverman proposed a natural measure of the post-critical complexity of a rational map defined over $\overline{\QQ}$.  Letting $\hhat_f:\overline{\QQ}\to\RR$ denote the usual canonical height function, we  define the critical height of $f(z)\in\overline{\QQ}(z)$ by
\[\hcrit(f)=\sum_{P\in \PP^1(\overline{\QQ})}\left(e_P(f)-1\right)\hhat_f(P),\]
where $e_P(f)$ is the ramification index of $f$ at $P$.
%(Note that the sum is supported on the critical points of $f$.)
Since $\hhat_f$ vanishes precisely on points which are preperiodic for $f$, and takes positive values otherwise, and since $e_P(F)>1$ if and only if $P$ is a critical point for $f$, it is immediate that $\hcrit(f)=0$ if and only if $f$ is post-critically finite.  This function is also invariant under change of coordinates, and so is well-defined as a function on $\mathcal{M}_d(\overline{\QQ})$.  It is possible, though, that this function is a height in name only; it has no immediately obvious relation to any height functions on $\mathcal{M}_d$ in the sense of Weil.  Our main result shows that, if we restrict attention to polynomials, such a relation does indeed exist.

 The case of quadratic polynomials is so straightforward in part because the moduli space is one-dimensional.  Indeed, it is not hard to show that a similar result holds for \emph{any} non-isotrivial one-parameter family whose generic fibre is not a Latt\`{e}s map, an observation which the author is certainly not the first to make.  Suppose that there exist a curve $C/\overline{\QQ}$, a  non-constant map $f:C\to\mathcal{M}_d$, and $t\in C(\overline{\QQ})$ of arbitrarily large height such that  the specialization $f_t$ is post-critically finite.  One may invoke the specialization theorem of Call and Silverman \cite{call-silv} to show, if $f_\eta$ is the generic fibre of the family, that
 \[\hcrit(f_\eta)=\lim_{h(t)\to\infty}\frac{\hcrit(f_t)}{h(t)}=0.\]
Since the family was assumed to be non-isotrivial, it follows from a result of Baker \cite{baker} that $\hcrit(f_\eta)=0$ only if the the generic fibre is post-critically finite.  But  Thurston's result mentioned above shows that any non-constant family $f:C\to \mathcal{M}_d$ of post-critically finite maps must be a family of Latt\`{e}s maps. 
   In the higher dimensional case, though, this argument fails because the results of \cite{call-silv} require the base to be a curve.  In the case of polynomials, we can use a more direct argument to avoid this heavy machinery.

We define the \emph{monic centred height} on $\mathcal{P}_d$, arguably the first height one would consider, as follows.  
First note that every polynomial with coefficients in $\overline{\QQ}$ is affine-conjugate to at least one polynomial of the form
\[z^d+a_{d-2}z^{d-2}+\cdots+a_0.\]
Although a given affine-conjugacy class might not be uniquely represented in this way, any two representatives are related by an affine conjugacy of the form $z\mapsto \zeta z$, for some $\zeta^{d-1}=1$.  In particular, the function
\[\hnaive{mc}(z^d+a_{d-2}z^{d-2}+\cdots+a_1z+a_0)=h(a_{d-2}, ...,  a_0),\]
is well-defined on conjugacy classes, where
\[h(a_{d-2}, ..., a_0)=\frac{1}{[E:\QQ]}\sum_{\sigma\in\Gal(E/\QQ)}\sum_{v\in M_\QQ}\log\max\{1, |a_{d-2}^\sigma|_v, ..., |a_0^\sigma|_v\}\]
for any Galois extension $E/\QQ$ containing $\QQ(a_{d-2}, ..., a_0)$.
Our main result relates the critical height to the monic centred height.  Note that, while we state this result over $\overline{\QQ}$, the proof also works in the case of function fields (where many of the error terms vanish).

\begin{theorem}\label{th:main}
There exist effectively computable constants  $C_1$ and $C_2$, both depending just on $d$, such that
\[\left(\frac{1}{2d-1}\right)\hnaive{mc}(f)-C_1\leq \hcrit(f)\leq 4\hnaive{mc}(f)+C_2\]
for all $f\in\mathcal{P}_d(\overline{\QQ})$.
\end{theorem}

It is clear from Theorem~\ref{th:main} that the set of post-critically finite polynomials of given degree is a set of bounded height (relative to the ample Weil height $\hnaive{mc}$ on $\mathcal{P}_d$).  Applying standard results about heights (see, for example \cite[Theorem~5.11]{jhs:aec}), one obtains the following consequence of this observation.
\begin{corollary}\label{cor:finite eff comp}
Fix $d\geq 2$ and $B\geq 1$.  Then there are only finitely many affine-conjugacy classes of post-critically finite polynomials of degree $d$, with coefficients of algebraic degree at most $B$, and a set of representatives of these conjugacy classes is effectively computable.
\end{corollary}

Indeed, the results are explicit enough that one could simply write down an expression in $d$ and $B$ which bounds the number of such conjugacy classes, but since this bound is far larger than the actual number, we have neglected to do so.  It is possible that the estimates in this article could be improved, but what limits  the quality of the upper bound are the rather large constants that arise from the best-known effective versions of Hilbert's Nullstellensatz.

Our second corollary is a much weaker version of Thurston's result on $\mathcal{M}_d$, but we provide a simple argument in the polynomial case in order to show in a self-contained way that our result applies to all complex post-critically finite polynomials.  It should be noted that Epstein \cite{epstein} has independently used similar techniques to prove a stronger version of Corollary~\ref{cor: all pcf maps are algebraic} in the case where $d$ is a prime power.
\begin{corollary}\label{cor: all pcf maps are algebraic}  The locus of post-critically finite maps in $\mathcal{P}_d$ is contained a countable union of $0$-dimensional $\QQ$-rational subvarieties.
In particular, if $f\in\mathcal{P}_d(\CC)$ is a post-critically finite class, then $f\in\mathcal{P}_d(\overline{\QQ})$.
\end{corollary}

It is natural to ask how close the inequalities in Theorem~\ref{th:main} are to being sharp. Before proceeding with the proof of the main result we present examples to show that the inequalities cannot be improved by too much.
For the first inequality, let
\[f_{c, d}(z)=z^d-\frac{dc}{d-1}z^{d-1},\]
which has a critical point of multiplicity $d-2$ at $z=0$ and one of multiplicity one at $z=c$.  
By Theorem~1 of \cite{pi:spec}, we have
\[\hcrit(f_{c, d})=\hhat_{f_{c, d}}(c)=h(c)+O(1),\]
where the implied constant depends only on $d$.
Conducting the appropriate change of variables, one can show that
\[\hnaive{mc}(f_{c, d})\geq dh(c)+O(1).\]
  Combining these estimates gives a family of examples witnessing
\[\liminf_{\substack{f\in \mathcal{P}_d(\overline{\QQ})\\\hnaive{mc}(f)\to\infty}}\frac{\hcrit(f)}{\hnaive{mc}(f)}\leq \frac{1}{d},\]
compared with the lower bound of $1/(2d-1)$ given by Theorem~\ref{th:main}.

For the other inequality, note that the polynomial $f_{c, d}(z)=z^d+c$ satisfies
\[\hcrit(f_{c, d})=(d-1)\hhat_{f_{c, d}}(0)=\frac{d-1}{d}h(c)+O(1),\]
again by Theorem~1 of \cite{pi:spec}, while clearly $\hnaive{mc}(f_{c, d})=h(c)$.  We have, then,
\[\limsup_{\substack{f\in \mathcal{P}_d(\overline{\QQ})\\\hnaive{mc}(f)\to\infty}}\frac{\hcrit(f)}{\hnaive{mc}(f)}\geq 1-\frac{1}{d}.\]  Theorem~\ref{th:main} bounds this quantity above by 4.

\section{Preliminaries and lemmas}

For the remainder of the paper, we will fix an integer $d\geq 2$ and an extension to $\overline{\QQ}$ of each of the usual absolute values on $\QQ$.  For $v\in M_\QQ$ and $f(z)\in \overline{\QQ}[z]$ of degree $d$, set
\[\lhat_{f, v}(z)=\lim_{N\to\infty}d^{-N}\log\max\{1, |f^N(z)|_v\},\]
 so that the usual canonical height corresponding to $f(z)$ may be defined by
\begin{equation}\label{eq:can height}\hhat_f(z)=\frac{1}{[E:\QQ]}\sum_{\sigma\in\Gal(E/\QQ)}\sum_{v\in M_\QQ}\lhat_{\sigma(f), v}(\sigma(z)),\end{equation}
for any Galois extension $E/\QQ$ containing $z$ and the coefficients of $f$.
We will also set
\[\lcrit{v}(f)=\sum_{f'(c)=0}\lhat_{f, v}(c),\]
where the sum is taken with multiplicity, and note that $\hcrit(f)$ may be defined in terms of these local contributions in a fashion similar to \eqref{eq:can height}.  Although we miss the infinite critical point in this sum, that point is fixed and hence has canonical height 0.  
Note that the archimedean contribution to the critical height has been studied extensively by DeMarco~\cite{demarco}.
   Finally, for convenience of notation, we define the symbol $(r)_v$ for any real number $r$ by
\[(r)_v=\begin{cases}r & \text{if }v\text{ is archimedean}\\ 1 & \text{otherwise.}\end{cases}\]

The following lemma tells us that, when $z$ is sufficiently $v$-adically close to the super-attracting fixed point $\infty$, the $v$-adic contributions to the canonical height and to the naive height are essentially the same.  \begin{lemma}\label{lem:basins}
Let
\[f(z)=a_dz^d+a_{d-1}z^{d-1}+\cdots+a_1z+a_0,\]
and let $|\cdot|_v$ be an absolute value on $\QQ$.  Then if
\[|z|_v>C_{f,v}=(2d)_v\max_{0\leq i\leq d}\left\{1, \left|\frac{a_i}{a_d}\right|_v^{1/(d-i)}, |a_d|_v^{-1/(d-1)}\right\},\]
then
\[\lhat_{f, v}(z)=\log|z|_v+\frac{1}{d-1}\log|a_d|_v+\epsilon(f, z, v),\]
where $\epsilon(f, z, v)=0$ if $v$ is non-archimedean, and
\[-\log 2\leq \epsilon(f, z, v)\leq \log\frac{3}{2}\]
otherwise.
\end{lemma}

\begin{proof}
See \cite[Lemma~5]{pi:spec}.  Note that the definition of the quantity $C_{f, v}$ here is not the same as that in \cite{pi:spec}, but the same proof works.
\end{proof}

From this point forward, we work with a particular normal form.  For $\mathbf{c}=(c_1, ..., c_{d-1})\in\AA^{d-1}(\overline{\QQ})$, we set
\[f_{\mathbf{c}}(z)=\frac{1}{d}z^d-\frac{1}{d-1}(c_1+\cdots+c_{d-1})z^{d-1}+\cdots+(-1)^{d-1}c_1c_2\cdots c_{d-1}z,\]
so that the critical points of $f_\mathbf{c}(z)$ are precisely $z=c_1$, $z=c_2$, and so on.

\begin{lemma}\label{lem:c_f to c-height}
Let $v\in M_\QQ$ and $\mathbf{c}\in\AA^{d-2}(\overline{\QQ})$.
There is a real number $\xi_v$ such that
\[\log C_{f_{\mathbf{c}}, v}\leq \log\max\{1, |c_1|, ..., |c_{d-1}|\}+\xi_v,\]
and furthermore $\xi_v=0$ for all but finitely many places $v\in M_\QQ$.
\end{lemma}

\begin{proof}
If we write
\[f_{\mathbf{c}}(z)=a_dz^d+a_{d-1}z^{d-1}+\cdots+a_1z+a_0,\]
then
 for each $0\leq i<d$, $s_i=\pm ia_i$ is the elementary symmetric polynomial of degree $d-i$ in $c_1, ..., c_{d-1}$, which is a sum of at most $d$ monomials.  It follows from the triangle/ultrametric inequality that
\begin{multline}\label{eq:triangle}\left|\frac{a_i}{a_d}\right|_v^{1/(d-i)}\leq \left|\frac{s_i/i}{1/d}\right|_v^{1/(d-i)}  \leq \left|\frac{d}{i}\right|_v^{1/(d-i)}(d)_v(\max\{|c_j|_v\}^{d-i})^{1/(d-i)}\\ \leq \left|\frac{d}{i}\right|_v^{1/(d-i)}(d)_v\max\{|c_j|_v\}.\end{multline}
If $|\cdot|_v$ is non-archimedean, and restricts on $\QQ$ to a $p$-adic absolute value with $p>d$, then, 
 $|d/i|_v=1$ and $|a_d|_v^{-1/(d-1)}=|d|_v^{1/(d-1)}=1$.  It follows that we may take $\xi_v=0$.
For the remaining non-archimedean places, it is clear from \eqref{eq:triangle} that we can take
\[\xi_v=\log\max\left\{\left|\frac{d}{i}\right|_v^{1/(d-i)}, 1\right\}\]
(since $|d|_v\leq 1$).  For the archimedean place, we take
\[\xi_v=\log\max\left\{\left|\frac{d}{i}\right|_v^{1/(d-i)}, |d|_v^{1/(d-1)} ,1\right\}+\log d.\]
\end{proof}

Next we prove a lemma which reduces the problem to relating the critical height of $f_{\mathbf{c}}$ to the height of the point $\mathbf{c}\in \AA^{d-1}(\overline{\QQ})$.

\begin{lemma}\label{lem:naive to c-height}
There exists reals $C_5$ and $C_6$, depending on $d$, such that for any $\mathbf{c}\in\AA^{d-1}(\overline{\QQ})$
\begin{equation}\label{eq:naive upper}\hnaive{mc}(f_{\mathbf{c}})\leq \sum_{i=1}^{d-1}h(c_i)+dh(c_1, ..., c_{d-1})+C_5,\end{equation}
and such that for any $\mathbf{c}\in\AA^{d-1}(\overline{\QQ})$ there exists a $\mathbf{c}'\in\AA^{d-1}(\overline{\QQ})$ with
\begin{equation}\label{eq:naive lower}\sum_{i=1}^{d-1}h(c_i')\leq 2\hnaive{mc}(f_{\mathbf{c}'})+C_6,\end{equation}
and with $f_{\mathbf{c}'}$ affine-conjugate to $f_{\mathbf{c}}$.
\end{lemma}

\begin{proof}
For the purpose of this proof, we will write $h(F)$ in place of $h(b_d, ..., b_{0})$, for any polynomial
\[F(z)=b_dz^d+\cdots+b_1z+b_0\in\overline{\QQ}[z].\]
  Note that if $F$ is monic and centred, then $h(F)$ and $\hnaive{mc}(F)$ coincide, but $\hnaive{mc}$ is an invariant of conjugacy classes, while $h$ is certainly not.
We first show that if $\psi(z)=\alpha z+\gamma\in \overline{\QQ}[z]$, and $F^\psi=\psi^{-1}\circ F\circ\psi$, then
\begin{equation}\label{eq:height affine}h(F^\psi)\leq h(F)+d (h(\alpha)+h(\gamma))+d\log 2+\log d.\end{equation}
To see this, note that if $F(z)=b_dz^d+\cdots+b_0$, then
\begin{eqnarray*}
F^\psi(z)&=&\alpha^{-1}\sum_{i=0}^d b_i(\alpha z+\gamma)^i-\alpha^{-1}\gamma\\
&=&\alpha^{-1}\sum_{i=0}^d\sum_{j=0}^i\left(b_i\binom{i}{j}(z\alpha)^j\gamma^{i-j}\right)-\alpha^{-1}\gamma\\
&=&\sum_{j=0}^d\alpha^{j-1}\left(\sum_{i=j}^d b_i\binom{i}{j}\gamma^{i-j}\right)z^j -\alpha^{-1}\gamma.
\end{eqnarray*}
In other words, if $F^\psi(z)=c_dz^d+\cdots +c_0$,
we have for $1\leq j\leq d$ the inequality
\begin{eqnarray*}
\log|c_j|_v&=&\log\left|\alpha^{j-1}\left(\sum_{i=j}^d b_i\binom{i}{j}\gamma^{i-j}\right)\right|_v\\
&\leq& (j-1)\log |\alpha|_v +\log\max_{j\leq i\leq d} \left| b_i\binom{i}{j}\gamma^{i-j}\right|_v+\log(d)_v\\
&\leq & (j-1)\log |\alpha|_v+\log\max_{j\leq i\leq d} |b_i|_v+ (d-1)\log\max\{1, |\gamma|_v\}\\ &&+\log\max_{1\leq j\leq i\leq d}\left|\binom{i}{j}\right|_v+\log(d)_v\\
&\leq & \log\max\{1, |b_d|_v, ..., |b_0|_v\}+(d-1)\log\max\{1, |\alpha|_v\}\\ &&+(d-1)\log\max\{1, |\gamma|_v\}+\log (d)_v+d\log (2)_v,
\end{eqnarray*}
since $|\binom{i}{j}|_v\leq 1$ for $v$ non-archimedean, and
\[\left|\binom{i}{j}\right|_v\leq2^i\leq 2^d\]
for $v$ archimedean.
Similarly,
\begin{eqnarray*}
\log|c_0|_v&=&\log\left|\alpha^{-1} \left(b_d\gamma^d+b_{d-1}\gamma^{d-1}+\cdots+b_1\gamma-\gamma+b_0\right)\right|_v\\
&\leq &\log|\alpha^{-1}|_v+\log\max\{1, |b_d|_v, ..., |b_0|_v\}+d\log\max\{1, |\gamma|_v\}\\ &&+\log (d+2)_v\\
&\leq &\log\max\{1, |b_d|_v, ..., |b_0|_v\}+\log\max\{1, |\alpha^{-1}|_v\}\\&&+d\log\max\{1, |\gamma|_v\}+\log (2d)_v.
\end{eqnarray*}
Combining these gives
\begin{multline*}
\log\max\{1, |c_d|_v, ..., |c_0|_v\}\leq \log\max\{1, |b_d|_v, ..., |b_0|_v\}\\+(d-1)\log\max\{1, |\alpha|_v\}+\log\max\{1, |\alpha^{-1}|_v\}\\+d\log\max\{1, |\gamma|_v\}+\log (d)_v+d\log (2)_v.
\end{multline*}
Summing over all places, and noting that $h(\alpha^{-1})=h(\alpha)$, gives the bound \eqref{eq:height affine}.

We will also use the fact (for example, see \cite[Theorem~VIII.5.9]{jhs:aec}) that if
\[z^d+b_{d-1}z^{d-1}+\cdots +b_1z+b_0=(z-\beta_1)(z-\beta_2)\cdots(z-\beta_d),\]
then
\begin{equation}\label{eq:height of roots}\sum_{i=1}^d h(\beta_i)-d\log 2\leq h(b_{d-1}, ..., b_0)\leq \sum_{i=1}^d h(\beta_i)+d\log 2.\end{equation}
%where the point $[b_d, ..., b_0]\in\PP^d$ is projective.  We will apply
%One checks, though, this height is related to heights of affine points by
%\begin{equation}\label{eq:proj to affine} h(b_{d-1}, ..., b_0)-h(b_d)\leq h([b_d, ..., b_0])\leq h(b_{d-1}, ..., b_0)+h(b_d).\end{equation}

Now, suppose that 
\[g(z)=z^d+a_{d-2}z^{d-2}+\cdots+a_1z+a_0\]
is affine-conjugate to $f_{\mathbf{c}}(z)$, so that $\hnaive{mc}(f_{\mathbf{c}})=h(g)$.  Now, for every fixed point $\gamma$ of $g$, if $\psi_\gamma(z)=d^{-1/(d-1)}z+\gamma$ (for some choice of $(d-1)$th root), then the polynomial $g^{\psi_\gamma}$ has the form $f_{\mathbf{c}'}$, for some $\mathbf{c}'\in\AA^{d-1}(\overline{\QQ})$, and necessarily this polynomial is affine conjugate to $f_\mathbf{c}$.  Now,  by \eqref{eq:height affine} we have
\begin{eqnarray*}h(g^{\psi_\gamma})&\leq &h(g)+d(h(d^{-1/(d-1)})+h(\gamma))+d\log 2+\log d\\
&\leq &h(g)+dh(\gamma)+O(d),
\end{eqnarray*}
where the implied constant is absolute.   Averaging over the various choices for $\gamma$ (taken with multiplicity), we obtain
\[\frac{1}{d}\sum_{g(\gamma)=\gamma}h(g^{\psi_\gamma})\leq h(g)+\sum_{g(\gamma)=\gamma}h(\gamma)+O(d).\]
But the $\gamma$ are roots of $g(z)-z$, and we claim that $h(g(z)-z)\leq h(g)+\log 2$.
Indeed, for any $v$ we have 
\begin{multline*}
\log\max\{1, |a_{d-2}|_v, ..., |a_1-1|_v, |a_0|_v\}\\ \leq \log\max\{1, |a_{d-2}|_v, ..., (2)_v\max\{1, |a_1|_v\}, |a_0|_v\}\\\leq \log\max\{1, |a_{d-2}|_v, ..., |a_1|_v, |a_0|_v\}+\log(2)_v,
\end{multline*}
and summing over all places gives the inequality $h(g(z)-z)\leq h(g)+\log 2$.
Since $g(z)-z$ is also monic,  we may apply \eqref{eq:height of roots} to obtain
\[ \frac{1}{d}\sum_{g(\beta)=\beta}h(g^{\psi_\beta})\leq h(g)+h(g(z)-z)+O(d)\leq 2h(g)+O(d),\]
and so for some choice of $\gamma$,  the polynomial $g^{\psi_\gamma}=f_{\mathbf{c}'}$ satisfies
\[h(f_{\mathbf{c}'})\leq  2\hnaive{mc}(g)+O(d)=2\hnaive{mc}(f_{\mathbf{c}'})+O(d).\]

Finally, we claim that for any polynomial $F$ with $F(0)=0$, we have
\[h(F')-\log \deg(F)\leq h(F)\leq h(F')+1.26\deg(F).\]
The first inequality is clear, but the second requires some work.  Writing $r=\deg(F)$, we note that if $F'(z)=b_{r-1}z^{r-1}+\cdots+b_0$, then
$F(z)=\frac{1}{r}b_{r-1}z^r+\cdots+b_0z$.  Now, if $v$ is any valuation, 
\begin{multline*}
\log\max\left\{1, \left|\frac{1}{r}b_{r-1}\right|_v, ..., |b_1|_v\right\}\leq\log\max\left\{1, \left|b_{r-1}\right|_v, ..., |b_1|_v\right\}\\+\log\max\left\{1, \left|\frac{1}{r}\right|_v, ..., \left|\frac{1}{2}\right|_v\right\}.\end{multline*}
Summing over all places yields
\[h(F)\leq h(F')+\sum_{p\text{ prime}}\max\{e\log p:e\in\ZZ\text{ and }p^e\leq r\}\leq h(F')+ \pi(r)\log r,\]
where $\pi(x)$ denotes the number of primes $p\leq x$.  The result now follows from the explicit estimate 
\[\pi(x)\leq \frac{1.26x}{\log x},\] for all $x>1$, due to Rosser and Schoenfeld \cite{rosser}. 
(Note that, using the more precise estimate in \cite{rosser}, one might replace the error term by the slightly better bound $r+3r/(2\log r)$.)

%(Note that, using recent improvements due to Dusart \cite{dusart}, one might replace the error term by an explicit estimate of the form $d(1+o(1))$.)
 
Since $f'_{\mathbf{c}'}$ is monic, \eqref{eq:height of roots}  gives
\begin{eqnarray*}
\sum_{i=1}^{d-1}h(c_i')
&\leq&h(f'_{\mathbf{c}'})+(d-1)\log 2\\
&\leq&h(f_{\mathbf{c}'})+(d-1)\log 2+\log d\\
&\leq & 2\hnaive{mc}(f_{\mathbf{c}'})+O(d).
\end{eqnarray*}
This gives \eqref{eq:naive lower}, with the additional observation that $C_6$ grows at most linearly in $d$.

For the other bound, we note that if $\psi(z)=\alpha z+\gamma$, for 
$\alpha=d^{1/(d-1)}$ and $\gamma=(d-1)^{-1}(c_1+\cdots+c_{d-1})$, then $f^\psi_{\mathbf{c}}$ is monic and centred, and hence $\hnaive{mc}(f_{\mathbf{c}})=h(f_{\mathbf{c}}^\psi)$.  Since $h(\alpha)=\frac{1}{d-1}\log d$, and
\[h(\gamma)\leq \log(d-1)+h(c_1+\cdots+c_{d-1})\leq h(c_1, ..., c_{d-1})+2\log (d-1),\]
we have by \eqref{eq:height affine} and \eqref{eq:height of roots} that
\begin{eqnarray*}
\hnaive{mc}(f_{\mathbf{c}})&=&h(f_{\mathbf{c}}^\psi)\\
&\leq &h(f_{\mathbf{c}})+d(h(\alpha)+h(\gamma))+O(d)\\
&\leq &h(f'_{\mathbf{c}})+dh(c_1, ..., c_{d-1})+O(d\log d)\\
&\leq &\sum_{i=1}^{d-1}h(c_i)+dh(c_1, ..., c_{d-1})+O(d\log d).
\end{eqnarray*}
We conclude that \eqref{eq:naive upper} holds, with $C_5=O(d\log d)$ as $d\to\infty$.
\end{proof}

The following lemma is crucial to the proof of Lemma~\ref{lem:main}.
% We let $R=\ZZ[\frac{1}{2}, \frac{1}{3}, ..., \frac{1}{d}]\subseteq \QQ$.
\begin{lemma}\label{lem:homog}
For each $1\leq i\leq d-1$, let $G_{i}(c_1, ..., c_{d-1})=f_{\mathbf{c}}(c_i)$.  Then the polynomials $G_i\in \QQ[c_1, ..., c_{d-1}]$ are homogeneous forms of degree $d$, with no common root (over $\overline{\QQ}$) other than the trivial root $c_i=0$ for all $i$.
\end{lemma}

\begin{proof}
The fact that $G_i$ is a homogeneous polynomial of degree $d$ follows immediately from the fact that the $i$th symmetric polynomial in $c_1, ..., c_{d-1}$ is homogeneous of degree $d-i$.  It remains to show that the $G_i$ have no common root, other than the trivial one, so we suppose that $(c_1, ..., c_{d-1})$ is some common root.  Now, if $z=c_i$ is a root of the polynomial $f_{\mathbf{c}}(z)$, then it is clearly a root of multiplicity
\begin{equation}\label{eq:root mult}1+\#\{1\leq j\leq d-1:c_j=c_i\},\end{equation}
since this is one more than the multiplicity of $z=c_i$ as a root of
\[f'_{\mathbf{c}}(z)=(z-c_1)(z-c_2)\cdots (z-c_{d-1}).\]
But for our chosen point $(c_1, ..., c_{d-1})\in \AA^{d-1}$, \emph{each} $c_i$ is a root of $f_{\mathbf{c}}(z)$, and so  by summing \eqref{eq:root mult} over distinct values of $c_i$, we see that the number of roots of $f_{\mathbf{c}}(z)$ (with multiplicity) is at least
\[\#\{c_1, ..., c_{d-1}\}+d-1.\]
The polynomial $f_{\mathbf{c}}(z)$ has degree $d$, and so it follows immediately that \[c_1=c_2=\cdots=c_{d-1},\]
and that this is the unique root of $f_{\mathbf{c}}(z)$.  However, we also have $f_{\mathbf{c}}(0)=0$, and so $c_i=0$ for all $i$.
\end{proof}

We now establish the main technical lemma in the proof of Theorem~\ref{th:main}.

\begin{lemma}\label{lem:main}
 For any $v\in M_\QQ$, there is a $\delta_v\geq 0$ such that
\begin{equation}\label{eq:main bound}\log\max\{1, |c_1|_v, ..., |c_{d-1}|_v\}\leq \lcrit{v}(f_\mathbf{c})+\delta_v\end{equation}
for all $\textbf{c}\in\AA^{d-1}(\overline{\QQ})$.
Furthermore, $\delta_v=0$ for all but finitely many $v\in M_\QQ$.
\end{lemma}

\begin{proof} 
Let \[G_i(c_1, ..., c_{d-1})=f_{\mathbf{c}}(c_i)\in\QQ[c_1, ..., c_{d-1}]\]
be the homogeneous polynomials defined in Lemma~\ref{lem:homog}.
By that lemma, these polynomials have no common root other than the trivial one, 
and so we may employ a standard argument (reproduced here for the convenience of the reader) to bound the $G_i$ from below.
By Hilbert's Nullstellensatz, the radical of the ideal generated by the $G_i$ is the ideal generated by all of the $c_i$.
  In particular there are homogeneous polynomials $F_{i,j}\in\QQ[c_1, ..., c_j]$ and $e\in\ZZ$ such that
\[c_i^e=F_{i,1}(\mathbf{c})G_1(\mathbf{c})+\cdots F_{i,(d-1)}(\mathbf{c})G_{d-1}(\mathbf{c}).\]
Note that $F_{i, j}$ has degree $e-d$.
For $v\in M_\QQ$, we define $\|F_{i, j}\|_v$ to be the maximum of the $v$-adic absolute values of the coefficients of $F_{i, j}$, so that
\[|F_{i, j}(c_1, ..., c_{d-1})|_v\leq \epsilon_v\|F_{i, j}\|_v\max\{|c_1|_v, ..., |c_{d-1}|_v\}^{e-d}\]
(here $\epsilon_v=1$ if the absolute value is ultrametric, and $\epsilon_v=\binom{e+1}{d-1}$ is the number of possible monomials of degree $e-d$ if $v$ is archimedean).
It follows that
\begin{eqnarray*}
e\log|c_i|_v&\leq& \log(d-1)_v+\log\max_{1\leq j\leq d-1}\{|F_{i, j}(\mathbf{c})G_j(\mathbf{c})|_v\}\\ 
&\leq&\log(d-1)_v+\log\max \{|G_j(\mathbf{c})|_v\}+\log\max_{1\leq j\leq d-1}\{|F_{i, j}(\mathbf{c})|_v\} \\
&\leq &
 \log(d-1)_v+\log\max|G_j(\mathbf{c})|_v+(e-d)\log\max\{|c_j|_v\}\\ &&+\log\max_{1\leq j\leq d-1}\|F_{i,j}\|_v+\log\epsilon_v.\end{eqnarray*}
 Since this holds for all $i$, we have
\begin{multline*}
e\log\max\{|c_1|_v, |c_2|_v, ..., |c_{d-1}|_v\}\leq \log\max\{|G_j(\mathbf{c})|_v\}\\ + (e-d)\log\max\{|c_1|_v, |c_2|_v, ..., |c_{d-1}|_v\}+B_v,\end{multline*}
where
\[B_v=\log(d-1)_v+\log\max\|F_{i,j}\|_v+\log\epsilon_v,\]
is clearly 0 for all but finitely many $v\in M_\QQ$.

Now, recalling that $G_j(c_1, ..., c_{d-1})=f_{\mathbf{c}}(c_j)$, we have
\[d\log\max\{|c_1|_v, |c_2|_v, ..., |c_{d-1}|_v\}\leq \log\max\{|f_{\mathbf{c}}(c_j)|_v\}+B_v.\]
Either $\max\{|c_1|_v, ..., |c_{d-1}|_v\}\leq 1$, in which case \eqref{eq:main bound} holds trivially, or else
$\max\{|c_1|_v, ..., |c_{d-1}|_v\}>1$.  In the latter case, we have 
\begin{equation}\label{eq:lower}d\log\max\{1, |c_1|_v, |c_2|_v, ..., |c_{d-1}|_v\}-B_v\leq \log|f_{\mathbf{c}}(c_j)|_v,\end{equation}
for some $j$ witnessing the maximum value of $|f_{\mathbf{c}}(c_j)|_v$.  By Lemma~\ref{lem:c_f to c-height},
this gives
\begin{equation}
\log C_{f_{\mathbf{c}}, v}+(d-1)\log\max\{1, |c_1|_v, |c_2|_v, ..., |c_{d-1}|_v\}-B_v-\xi_v\leq \log |f(c_j)|_v.\end{equation}
Now, for all but finitely many non-archimedean places $v$, we have $B_v=0$ and $\xi_v=0$, and so this immediately implies
\[\log C_{f_{\mathbf{c}}, v}<\log  |f(c_j)|_v,\]
since $d\geq 2$ and since we have assumed that $\max\{|c_1|_v, ..., |c_{d-1}|_v\}>1$.  From this we obtain
\[d\log\max\{1, |c_1|_v, ..., |c_{d-1}|_v\}\leq \log|f(c_j)|_v=\lhat_{f, v}(f(c_j))\leq d\lcrit{v}(f),\]
by Lemma~\ref{lem:basins}.

The remaining cases are not particularly different.  In general, if $\log|f(c_j)|_v$ is at most $\log C_{f_{\mathbf{c}}, v}$ then \eqref{eq:lower} and Lemma~\ref{lem:c_f to c-height} give
\[d\log\max\{1, |c_1|_v, |c_2|_v, ..., |c_{d-1}|_v\}-B_v\leq \log\max\{1, c|_1|_v, ..., |c_{d-1}|_v\}+\xi_v,\]
and so \eqref{eq:main bound} holds trivially as long as
\[\delta_v\geq \frac{1}{d-1}(B_v+\xi_v).\]
If, on the other hand, we have $\log|f(c_j)|_v> \log C_{f_{\mathbf{c}}, v}$, then Lemma~\ref{lem:basins} and \eqref{eq:lower} give
\begin{eqnarray*}
d\log\max\{1, |c_1|_v, |c_2|_v, ..., |c_{d-1}|_v\}&\leq& \lhat_{f, v}(f(c_j))+\frac{1}{d-1}\log|d|_v\\&&+\left(\log \frac{3}{2}\right)_v+B_v\\&\leq& d\lcrit{v}(f)+\frac{1}{d-1}\log|d|_v\\&&+\left(\log \frac{3}{2}\right)_v+B_v.
\end{eqnarray*}
This again shows \eqref{eq:main bound} as long as \[\delta_v\geq \frac{1}{d}\left(\frac{1}{d-1}\log|d|_v+\left(\log \frac{3}{2}\right)_v+B_v\right).\]
\end{proof}

Note that Lemma~\ref{lem:main} is purely local.  If $\QQ$ were replaced by any valued field, then Lemma~\ref{lem:main} shows that the locus of post-critically bounded maps, in the $c_i$ coordinates, is a bounded subset of moduli space.  Modifying the proof of Lemma~\ref{lem:naive to c-height}, in which every step is obtained by summing local heights, we get the same for the $a_i$ coordinates.  Indeed, it is exactly this that we use to prove Corollary~\ref{cor:finite eff comp}, with the additional observation that over a number field, this bounded set is a ball of radius one for all but finitely many places.

\section{Proof of the main results}

We now prove Theorem~\ref{th:main} from the lemmas above.  The first inequality, which is the more interesting part, follows by summing \eqref{eq:main bound} of Lemma~\ref{lem:main} over all places of $\QQ$.  In particular, if $(c_1, ..., c_{d-1})\in\AA^{d-1}(E)$, for $E/\QQ$ Galois, we have
\begin{eqnarray}
h(c_1, ..., c_{d-1})&=&\frac{1}{[E:\QQ]}\sum_{\sigma\in\Gal(E/\QQ)}\sum_{v\in M_\QQ}\log\max\{1, |c_1^\sigma|_v, ..., |c_{d-1}^\sigma|_v\}\nonumber\\
&\leq&\frac{1}{[E:\QQ]}\sum_{\sigma\in\Gal(E/\QQ)}\sum_{v\in M_\QQ}\left( \lcrit{v}(f_{\mathbf{c}}^\sigma)+\delta_v\right)\nonumber\\
&=&\hcrit(f)+\sum_{v\in M_\QQ}\delta_v,\label{eq:c-height crit height}
\end{eqnarray}
an inequality which does not depend on the choice of $E$.
We can now combine \eqref{eq:c-height crit height} with Lemma~\ref{lem:naive to c-height} to obtain 
\begin{eqnarray*}
\hnaive{mc}(f)&\leq &\sum_{i=1}^{d-1}h(c_i)+dh(c_1, ..., c_{d-1})+O(1)\\
&\leq &(2d-1)h(c_1, ..., c_{d-1})+O(1)\\
&\leq& (2d-1) \hcrit (f)+O(1)
\end{eqnarray*}
where the implied constant depends only on $d$.  

The upper bound is much more elementary, and uses well-known techniques.  In particular, one can use the triangle inequality and a standard telescoping sum argument to show that for any polynomial $f(z)=\sum a_i z^i$ and any $z\in\overline{\QQ}$,
\[\hhat_{f}(z)\leq h(z)+ \frac{1}{d-1} h(a_{d}, ..., a_0)+\frac{1}{d-1}\log (d+1).\]
Estimating the coefficients of $f_{\mathbf{c}}$ as in the proof of Lemma~\ref{lem:naive to c-height}, we see that
\[\hhat_{f_{\mathbf{c}}}(z)\leq h(z)+  \frac{1}{d-1}\sum_{j=1}^{d-1}h(c_j)+O(1),\]
where the implied constant depends on $d$.
 So, it follows that
\begin{eqnarray*}
\hcrit(f_{\mathbf{c}})&\leq &\sum_{i=1}^{d-1} \left(h(c_i)+ \frac{1}{d-1}\sum_{j=1}^{d-1}h(c_j)+O(1)\right)\\
&=&2\sum_{i=1}^{d-1} h(c_i)+O(1),
\end{eqnarray*}
for any $\mathbf{c}\in \AA^{d-1}(\overline{\QQ})$.  By Lemma~\ref{lem:naive to c-height}, there is a $\mathbf{c}'$ with $f_\mathbf{c}$ affine-conjugate to $f_{\mathbf{c}'}$, and
\begin{eqnarray*}
\hcrit(f_{\mathbf{c}'})&\leq &2\sum_{i=1}^{d-1} h(c_i')+O(1)\\
&\leq &4\hnaive{mc}(f_{\mathbf{c}'})+O(1)
\end{eqnarray*}
and, since both $\hcrit$ and $\hnaive{mc}$ are well-defined on conjugacy classes, we have
\[\hcrit(f_\mathbf{c})\leq 4\hnaive{mc}(f_\mathbf{c})+O(1).\]
This concludes the proof of Theorem~\ref{th:main}.

%In fact, our results are slightly sharper if we use a somewhat less natural height on $\mathcal{P}_d$.
%Specifically, every polynomial is affine-conjugate to some $f_{\mathbf{c}}$ with leading coefficient $\frac{1}{d}$, a fixed point at $0$, and critical points $c_1, ..., c_{d-1}$, and there are generically $d$ such polynomials in each conjugacy class (corresponding to the $d$ fixed points which might be identified with the origin).  If we then define the \emph{critical point height}  $\hnaive{cp}(f)$ to be the average of $dh(c_1, ..., c_{d-1})$ over these $d$ representatives, we obtain another height on $\mathcal{P}_d$ coming from another cover  of $\mathcal{P}_d$ by $\AA^{d-1}$.
%The proof above confirms the existence of constants $C_3$ and $C_4$, depending on $d$, such that
%\[\hnaive{cp}(f)-C_3\leq \hcrit(f)\leq 2\hnaive{cp}(f)+C_4.\]
% really, 2(d-1)/d

Corollary~\ref{cor:finite eff comp} follows almost immediately from this.  If 
\[f(z)=a_dz^d+a_{d-1}z^{d-1}+\cdots+a_0\in\overline{\QQ}[z]\]
with $[\QQ(a_d, ..., a_0):\QQ]\leq B$, then $f$ is affine-conjugate to some monic, centred polynomial $g(z)$ with coefficients in $\QQ(a_d^{1/(d-1)}, ..., a_0)$.
  If $f(z)$ is post-critically finite, then so is $g(z)$, and hence by Theorem~\ref{th:main}  the coefficients of $g(z)$ lie in a set of bounded height, and algebraic degree at most $(d-1)B$, in $\AA^{d-1}(\overline{\QQ})$, which must be a finite, effectively computable set.  The one subtlety is that one may effectively decide which of the polynomials with coefficients in this set is actually post-critically finite, which amounts to finding an upper bound on the orbit size of a critical point of such a polynomial.  But the critical points themselves will be algebraic numbers of degree at most $(d-1)^2B$, and will also be contained in a set of bounded height, and so the finiteness of this set gives an effectively computable upper bound on the possible orbit sizes of these critical points.  This gives an effective algorithm for deciding which of these polynomials is actually post-critically finite.

To prove Corollary~\ref{cor: all pcf maps are algebraic}, we suppose otherwise.  Since the points in $\mathcal{P}_d(\CC)$ corresponding to post-critically finite polynomials are clearly contained in the union of countably many $\QQ$-rational affine subvarieties, defined by the different possible orbit types of the critical points, we suppose that one of these subvarieties $V$ contains a transcendental point.  It follows that $V$ contains a curve defined over some number field $L$, say, since there is a map $\QQ(V)\to\CC$ which doesn't factor through a map to $\overline{\QQ}$, and so there is a surjective map from $\QQ(V)$ to a ring with transcendence rank one over $\QQ$, which we may take to be the function field of a curve $X\subseteq V$ defined over a number field $L$.
%In other words, the image of this map, $B$ say, has positive transcendance rank over $\QQ$, and without loss of generality we take this rank to be 1.  But then $B\cong L(X)$ is isomorphic to the function field of a curve defined over some finite extension of $\QQ$, and there is a non-constant map $X\to  V$ defined over $L$, dominating some curve $Y\subseteq V$.  
This curve admits a non-constant map to $\PP^1$ of some degree $D$, and pulling back the $L$-rational points on $\PP^1$, we obtain infinitely many points on $X(\overline{\QQ})\subseteq V(\overline{\QQ})$ of algebraic degree at most $D[L:\QQ]$, all of which correspond to post-critically finite polynomials.  This contradicts Corollary~\ref{cor:finite eff comp}, and so it must be the case that all of the varieties $V$ are $0$-dimensional, and hence all of the points in $V(\CC)$ algebraic.

\section{Computations and examples}

Theorem~\ref{th:main} can be made completely effective, for example by invoking the effective version of the Nullstellensatz due to Masser and W\"{u}sthotlz \cite{mw}, but the resulting constants (which depend on $d$) are far too large to be of computational use.  Since it is the effective Nullstellensatz that is the limiting component of the argument, we have not made any effort of optimize the remaining estimates in this paper.  Instead, we will use the ideas of Theorem~\ref{th:main}, rather than the proof itself, to list the post-critically finite monic cubic polynomials with coefficients in $\QQ$.

One can check rather easily that if $c\in\QQ$ and $z^2+c$ is post-critically finite, then $c\in\{0, -1, -2\}$.  We claim that if $f(z)=z^3+Az+B$ has coefficients in $\QQ$ and is post-critically finite, then
\[(A, B)\in \left\{\Big(-3, 0\Big), \Big(-\frac{3}{2}, 0\Big), \Big(-\frac{3}{4}, \frac{3}{4}\Big), \Big(-\frac{3}{4},-\frac{3}{4}\Big),\Big(0, 0\Big), \Big(\frac{3}{2}, 0\Big), \Big(3, 0\Big)\right\}.\]
The proof of this is based on the proof of Theorem~\ref{th:main}, but is much more efficient as we can replace the Nullstellensatz with an explicit elimination.  We will suppose that all of the absolute values on $\QQ$ have been extended in some way to $\overline{\QQ}$, that $A, B\in \QQ$, and that we have chosen $\alpha, \beta\in \overline{\QQ}$ with $A=-3\alpha^2$ and $B=\beta^3$.  Note that the critical points of $f(z)$ are $z=\pm\alpha$.  

One checks that we may take
\[C^*_{f, v}=(2)_v\max\{1, |A|^{1/2}_v, |B|^{1/3}_v\},\]
and obtain  that $z$ is preperiodic only if $|z|_v\leq C^*_{f, v}$ for all $v\in M_{\QQ}$.  Note that this is a slight improvement on the relevant case of Lemma~\ref{lem:basins}, since $C_{f, \infty}=3C^*_{f, \infty}$.  In particular, at the archimedean place, the condition $|z|>C^*_{f, \infty}$ gives
\[|Az+B|< \frac{1}{4}|z|^3+\frac{1}{8}|z|^3<\frac{1}{2}|z|^3,\]
whereupon $|f(z)|\geq \frac{1}{2}|z|^3$.  This is enough to ensure $|f(z)|\geq C^*_{f, \infty}$ and, by induction,
\[3^{-N}\log|f^N(z)|\geq \log|z|-\frac{1-3^{-N}}{2}\log 2.\]
Taking $N\to \infty$, we obtain
\[\hhat_f(z)\geq\lhat_{f, \infty}(z)> \frac{1}{2}\log 2.\]
At the non-archimedean places, the condition $|z|_v>C^*_{f, v}$ implies $|Az+B|_v<|z|_v^3$, and so $|f(z)|_v=|z|_v^3$.  By induction, we obtain in this case $\lhat_{f, v}(z)=\log|z|_v>0$.  So $f(z)=z^3-3\alpha^2z+\beta^3$ is postcritically finite only if \[|f^N(\pm \alpha)|_v\leq\log C^*_{f, v}\] for all $N$ and all $v\in M_\QQ$.

Now, suppose that $v\nmid \infty$ and $v\nmid 6$.  We have
\[C^*_{f, v}=\max\{1, |\alpha|_v, |\beta|_v\}\]
and we can only have $f(z)$ post-critically finite (indeed, $v$-adically post-critically bounded) if
\[\max\{|f(\alpha)|_v, |f(-\alpha)|_v\}\leq C^*_{f, v}=\max\{1, |\alpha|_v, |\beta|_v\}.\]
Since $f(\alpha)+f(-\alpha)=2\beta^3$ and $f(\alpha)-f(-\alpha)=4\alpha^3$, the ultrametric inequality yields
\[\max\{|\alpha|_v^3, |\beta|_v^3\}\leq \max \{|f(\alpha)|_v, |f(-\alpha)|_v\}\leq \max\{1, |\alpha|_v, |\beta|_v\},\]
which is impossible unless $\max\{|\alpha|_v, |\beta|_v\}\leq 1$.  We have shown that $\alpha$ and $\beta$, and hence $A$ and $B$, are integral except possibly at places above $6$.
  The argument at $p\mid 6$ is nearly identical, and show that $4A, 8B\in \ZZ$.  The triangle inequality gives a similar estimate for the archimedean absolute value on $\QQ$, which turns out to yield 
\[|A|\leq 3^{3/2}\qquad |B|\leq  3^{9/4}.\]

In other words, if $z^3+Az+B$ is post-critically finite, then it is the case that $A=\frac{a}{4}$ for some $a\in\{-20, -19, ..., 20\}$ and $B=\frac{b}{8}$ for some $b\in \{-94, -33, ..., 94\}$.
Furthermore, since $z^3+Az+B$ is affine-conjugate to $z^3+Az-B$, we need only treat positive values of $B$, leaving just 3895 cubic polynomials to consider.  A quick computation in Pari shows that of these, all but 86 have $|f^N(\pm\alpha)|_\infty>C^*_{f, \infty}$ for some $N\leq 14$, which is enough to ensure that the critical point is not preperiodic.  Of the remaining 86, all but those listed above have a critical point which escapes 2-adically, a fact which may be observed on a case-by-case basis (the author used Maple to check this).

\end{document}